\documentclass{amsart}
\usepackage{thm-restate}
\usepackage{hyperref}

\let\P\relax
\newcommand{\P}{\mathbf{P}}
\newcommand{\N}{\mathbf{N}}
\newcommand{\Z}{\mathbf{Z}}
\newcommand{\Q}{\mathbf{Q}}
\newcommand{\R}{\mathbf{R}}

\newcommand{\E}{\mathbf{E}}
\newcommand{\tto}{\Rightarrow}
\newcommand{\eps}{\epsilon}
\DeclareMathOperator*{\esssup}{ess\,sup}
\DeclareMathOperator*{\essinf}{ess\,inf}
\newcommand{\inv}[1]{F^{\leftarrow}_{#1}}

\newtheorem{theorem}{Theorem}
\newtheorem{proposition}[theorem]{Proposition}

\newtheorem{lemma}[theorem]{Lemma}

\begin{document}

\title[CLT via Wasserstein Metric]{A Proof of the Central Limit Theorem using the $2$-Wasserstein metric}
\author{Calvin Wooyoung Chin}
\begin{abstract}
We prove the Lindeberg--Feller central limit theorem without using characteristic functions or Taylor expansions, but instead by measuring how far a distribution is from the standard normal distribution according to the $2$-Wasserstein metric.
This falls under the category of renormalization group methods.
The facts we need about the metric are explained and proved in detail.
We illustrate the idea on a classical version of the central limit theorem before going into the main proof.

\end{abstract}
\maketitle

\section{Introduction}
Let $X$ and $Y$ be independent random variables with mean $0$ and variance $1$.
The central question of this note is whether the distribution of
\[
\frac{X+Y}{\sqrt{2}}
\]
is ``closer" to the standard normal than those of $X$ and $Y$ in some appropriate sense.
Assuming this is true, the idea is that if $X_1,X_2,\dots$ are independent with mean $0$ and variance $1$, then
\[
\frac{X_1+X_2+X_3+X_4}{\sqrt{4}} = \frac{\frac{X_1+X_2}{\sqrt{2}} + \frac{X_3+X_4}{\sqrt{2}}}{\sqrt{2}}
\]
would be closer to the standard normal than $(X_1+X_2)/\sqrt{2}$ and $(X_3+X_4)/\sqrt{2}$, which are in turn closer than $X_1,X_2,X_3$, and $X_4$.
We might repeat this to show that weighted averages of $8$ terms, $16$ terms, $\dots$ are increasingly closer to the standard normal, hopefully leading to a version of the central limit theorem (CLT).

This idea, referred to as the renormalization group approach, is well-known and capable of proving the CLT; see \cite{Ott23} for a nice summary of the literature.
When the CLT is proven in this way, the CLT itself is often a test bed, and the actual goal is to apply the same idea to harder problems.
In this note, however, we would like to put a little more emphasis on making the proof of the CLT accessible.
Our approach falls under the category that defines an actual topological metric between distributions.
Other such approaches include \cite{Nei04, Ott23}, where the Zolotarev metric \cite{Zol76} or a metric based on the characteristic function are used.

Before introducing the Zolotarev metric in \cite{Nei04}, the authors briefly discuss the $2$-Wasserstein metric, but choose not to use it because it is not suited for the application of the Banach contraction method.
The $2$-Wasserstein distance between (the distributions of) $X$ and $Y$ is given by
\[
W_2(X,Y) := \inf_{X',Y'}\sqrt{\E\bigl[(X'-Y')^2\bigr]},
\]
where $X'$ and $Y'$ range over random variables with the same distributions as $X$ and $Y$.
The point is that the dependence or correlation between $X'$ and $Y'$, i.e.\ the coupling, is not specified.


The Wasserstein metric is a well-known metric in probability theory that appears in the context of optimal transport and the Wasserstein generative adversarial networks (WGANs) in machine learning, for example.
There is a $p$-Wasserstein metric corresponding to the $L^p$ metric for every $p \in [1,\infty]$.
In this note, we use the $2$-Wasserstein metric to prove the central limit theorem.

An important property of the $2$-Wasserstein metric is the following.
Throughout this note, $\tto$ denotes convergence in distribution.

\begin{restatable}{proposition}{wimpliesdist} \label{prop:W2_implies_dist}
Let $X,X_1,X_2,\dots$ be random variables with mean~$0$ and variance~$1$.
If $W_2(X_n,X) \tto 0$, then $X_n \tto X$.
\end{restatable}

\begin{restatable}{proposition}{distimpliesw} \label{prop:dist_implies_W2}
Let $X,Y,X_1,X_2,\dots$ be random variables with mean~$0$ and variance~$1$.
If $X_n \tto X$, then $W_2(X_n,Y) \to W_2(X,Y)$.
\end{restatable}

Thanks to these facts, we can prove the central limit theorem by examining the $2$-Wasserstein distances.
Our key result is the following. Throughout the note, let $Z$ be a standard normal random variable.

\begin{restatable}{theorem}{averagecloser} \label{thm:average_closer}
If $X$ and $Y$ are independent random variables with mean $0$ and variance $1$, then
\[
W_2\biggl(\frac{X+Y}{\sqrt{2}}, Z\biggr)^2 \le \frac{W_2(X,Z)^2 + W_2(Y,Z)^2}{2}.
\]
The equality holds if and only if both $X$ and $Y$ are standard normal.
\end{restatable}

The condition for equality is the nontrivial part.
Theorem~\ref{thm:average_closer} is powerful enough to imply the following theorem that says the sum of small independent random variables has a distribution that is close to the normal, explaining why the normal distribution is so ubiquitous.

\begin{restatable}[Lindeberg--Feller]{theorem}{lindebergfeller}\label{thm:lindeberg-feller}
For each $\eps > 0$, let $M_\eps$ be the supremum of
\[
W_2\Bigl(\sum_{j=1}^n X_j, Z\Bigr),
\]
where $X_1,\dots,X_n$ ranges over any finite sequence of independent mean-zero random variables with $|X_j| \le \eps$ for $j=1,\dots,n$ and $\sum_{j=1}^n \E [X_j^2] = 1$.
Then $M_\eps \to 0$ as $\eps \to 0$.
\end{restatable}

This is one way to state the Lindeberg--Feller theorem \cite[Theorem~3.4.10]{DurPTE}, which is the most general central limit theorem one typically sees. After the theorem in \cite{DurPTE}, there is a brief remark on why this implies the usual Lindeberg--Levy~CLT.

Our approach to the central limit theorem is more quantitative or ``soft analytic" than many classical proofs.
Unlike the standard proof \cite[Theorem~3.4.1]{DurPTE} using characteristic functions or the proof \cite[Subsection~2.2.3]{TaoRM} using the moment method, this proof avoids the use of Fourier analysis.
Other such approaches include the one by Trotter \cite{Tro59} that replaces characteristic functions with linear operations on some function space, and the proof by the so-called Lindeberg swapping \cite{Lin22}.
Unlike those proofs, the one we present here does not even use Taylor expansions.


In Section~\ref{sec:lindeberg_levy}, we will use Theorem~\ref{thm:average_closer} to complete the proof idea we sketched in the beginning of this introduction.
This part is not new, and similar proofs can be found in the literature cited above.
After introducing the notion of inverses of cumulative distribution functions in Section~\ref{sec:2-Wasserstein}, we will prove Theorem~\ref{thm:average_closer} in Section~\ref{sec:average_closer}.
In Section~\ref{sec:lindeberg-feller} that follows, we will prove the Lindeberg--Feller theorem (Theorem~\ref{thm:lindeberg-feller}).

Any proof of the central limit theorem involves some measure theoretic probability.
In our case, Proposition~\ref{prop:W2_implies_dist} and~\ref{prop:dist_implies_W2} and a couple other facts fall under this category.
Since the proofs of these are rather standard and do not contain new idea, we collected them at the end in Section~\ref{sec:measure_proofs} for those who are interested.

\section{Proof of a version of the Lindeberg--L\'evy CLT} \label{sec:lindeberg_levy}

To understand how one can use Theorem~\ref{thm:average_closer} to prove the central limit theorem, let us consider the following easier version.
Let $S_n := X_1 + \dots + X_n$.

\begin{theorem}[Lindeberg--L\'evy, bounded, lacunary] \label{thm:lindeberg_levy_bounded_lacunary}
Let $X_1,X_2,\dots$ be i.i.d.\ random variables with $|X_1| \le B < \infty$, $\E X_1 = 0$, and $\E X_1^2 = 1$.
Then,
\[
\frac{S_{2^n}}{\sqrt{2^n}} \tto Z
\]
where $Z$ is a standard normal random variable.
\end{theorem}

From this, it is not difficult to derive the usual Lindeberg--L\'evy theorem where $X_1$ can be unbounded and $2^n$ is replaced with $n$.
However, we omit the detail as we will prove the stronger Theorem~\ref{thm:lindeberg-feller}.

We need the following two lemmas, which are proved in Section~\ref{sec:measure_proofs}.

\begin{restatable}{lemma}{convsubseq} \label{lem:conv_subseq}
If there is a $B < \infty$ with $\E X_n^2 \le B$ for all $n\in\N$, then some subsequence of $(X_n)_{n\in\N}$ converges in distribution.
\end{restatable}

\begin{restatable}{lemma}{momentconv} \label{lem:bdd_4_moment}
If $X_n \tto X$ and there is a $B < \infty$ with $\E X_n^4 \le B$ for all $n\in\N$, then $\E X_n^2 \tto \E X^2$ and $\E X_n \to \E X$.
\end{restatable}

\begin{proof}[Proof of Theorem~\ref{thm:lindeberg_levy_bounded_lacunary}]
By Theorem~\ref{thm:average_closer}, the sequence
\[
\biggl(W_2\biggl(\frac{S_{2^n}}{\sqrt{2^n}}, Z\biggr) \biggr)_{n\in\N}
\]
is nonincreasing. We want to show that this converges to $0$.

Since
\begin{equation} \label{eq:s2nk_tto_l}
\E\biggl[ \biggl(\frac{S_n}{\sqrt{n}}\biggr)^4 \biggr]
= \frac{n\E X_1^4 + 3n(n-1)(\E X_1^2)^2}{n^2} \le B^4 + 3
\end{equation}
for all $n\in\N$, lemma~\ref{lem:conv_subseq} and lemma~\ref{lem:bdd_4_moment} imply that for some subsequence $n_1 < n_2 < \dots$, we have
\[
\frac{S_{2^{n_k}}}{\sqrt{2^{n_k}}} \tto L \qquad \text{as $k\to\infty$,}
\]
where $L$ is some random variable with mean~$0$ and variance~$1$.

By Proposition~\ref{prop:dist_implies_W2}, we have
\begin{equation} \label{eq:w2lz_limit}
W_2(L,Z) = \lim_{k\to\infty} W_2\biggl(\frac{S_{2^{n_k}}}{\sqrt{2^{n_k}}}, Z\biggr) = \inf_{n\in\N}  W_2\biggl(\frac{S_{2^n}}{\sqrt{2^n}}, Z\biggr).
\end{equation}
If $L'$ is an independent copy of $L$, then
\[
\frac{S_{2^{n_k+1}}}{\sqrt{2^{n_k+1}}} \tto \frac{L+L'}{\sqrt{2}},
\]
and thus
\begin{equation} \label{eq:w2llpz_limit}
W_2\biggl(\frac{L+L'}{\sqrt{2}}, Z\biggr) = \lim_{k\to\infty} W_2\biggl(\frac{S_{2^{n_k+1}}}{\sqrt{2^{n_k+1}}}, Z\biggr) = \inf_{n\in\N}  W_2\biggl(\frac{S_{2^n}}{\sqrt{2^n}}, Z\biggr).
\end{equation}

By \eqref{eq:w2lz_limit} and \eqref{eq:w2llpz_limit}, we have
\[
W_2\biggl(\frac{L+L'}{\sqrt{2}}, Z\biggr)^2
= \frac{W_2(L,Z)^2 + W_2(L',Z)^2}{2}.
\]
Theorem~\ref{thm:average_closer} then implies that $L$ is standard normal.
Since
\[
\lim_{n\to\infty} W_2\biggl(\frac{S_{2^n}}{\sqrt{2^n}}, Z\biggr)
= \inf_{n\in\N} W_2\biggl(\frac{S_{2^n}}{\sqrt{2^n}}, Z\biggr) = W_2(L,Z) = 0,
\]
we have $S_{2^n}/\sqrt{2^n} \tto Z$.
\end{proof}

\section{Inverses of CDFs and the $2$-Wasserstein metric} \label{sec:2-Wasserstein}

For a random variable $X$, let us denote its cumulative distribution function by $F_X \colon \R \to [0,1]$.
We say that $(X_n)_{n\in\N}$ converges in distribution to $X$ and write $X_n \tto X$ if
\[ F_{X_n}(t) \to F_X(t) \qquad \text{as $n \to \infty$} \]
for all $t \in \R$ where $F_X$ is continuous.

However, it is not so easy to use this definition directly to prove convergence in distribution. As a result, one often uses some equivalent condition that is easier to show.
Typically the Portmanteau theorem \cite[Theorem~8.4.1]{ResProb} provides such a condition, but in this note we will instead use \emph{inverses} of cumulative distribution functions.
These are well-known tools in probability, and the relevant proofs are rather elementary.
We will mention where one can find proofs whenever we omit some details.

If $F_X$ is a strictly increasing function, then its inverse $F_X^{-1} \colon (0,1) \to \R$ has an interesting property: if we view this as a random variable defined on the sample space $(0,1)$, then it has the same distribution as $X$. To see this, notice that
\[
\P(F_X^{-1} \le x) = F_X(x) = \P(X \le x) \qquad \text{for all $x\in\R$.}
\]

A similar thing can be done even if $F_X$ is not strictly increasing.
In general, we consider the ``inverse" $\inv{X} \colon (0,1) \to \R$ given by
\[
\inv{X}(t) := \inf\{x \in \R : F_X(x) \ge t\}.
\]
It can be shown that $\inv{X}$ is nondecreasing, left-continuous, and has the same distribution as $X$ if we view it as a random variable defined on $(0,1)$. For a proof, see \cite[2.5.2]{ResProb}.

An important property of inverses for the purposes of this note is the following.

\begin{restatable}{proposition}{rearr} \label{prop:rearr}
For any random variables $X$ and $Y$ with finite variances, we have
\begin{equation} \label{eq:rearr}
\E[XY] \le \int_0^1 \inv{X}\inv{Y}.
\end{equation}
\end{restatable}

The full proof of Proposition~\ref{prop:rearr} can be found in Section~\ref{sec:measure_proofs}, but the basic idea is simple:
it is nothing more than a continuous version of the rearrangement inequality.
The inequality says that if $x_1 \le \dots \le x_n$ and $y_1 \le \dots \le y_n$, then
\[
x_1y_{\sigma(1)} + \dots + x_ny_{\sigma(n)}
\le x_1y_1 + \dots + x_ny_n
\]
for any permutation $\sigma \colon \{1,\dots,n\} \to \{1,\dots,n\}$.
The pair $(\inv{X},\inv{Y})$is a way to couple $X$ and $Y$ so that $X$ increases as $Y$ increases.

Thanks to Proposition~\ref{prop:rearr}, we can express $W_2(X,Y)$ in terms of $\inv{X}$ and $\inv{Y}$.

\begin{restatable}{corollary}{wasschar} \label{cor:W2_inv}
For any random variables $X$ and $Y$ with mean $0$ and variance $1$, we have
\[
W_2(X,Y)^2 = \int_0^1 (\inv{X} - \inv{Y})^2.
\]
\end{restatable}

\begin{proof}
For any $X'$ and $Y'$ with the same distribution as $X$ and $Y$, we have
\[
\E\bigl[(X'-Y')^2\bigr] = 2 - 2\E[X'Y'] \ge 2 - 2\int_0^1 \inv{X}\inv{Y} = \int_0^1 \bigl(\inv{X} - \inv{Y}\bigr)^2.
\]
Thus, the infimum of the possible $\E\bigl[(X'-Y')^2\bigr]$ is $\int_0^1 \bigl(\inv{X} - \inv{Y}\bigr)^2$.
\end{proof}

\section{Proof of Theorem~\ref{thm:average_closer}} \label{sec:average_closer}

\averagecloser*

\begin{proof}
Let $f_1 = F_{X/\sqrt{2}}^\leftarrow$, $f_2 = F_{Y/\sqrt{2}}^\leftarrow$, and $g = F_{Z/\sqrt{2}}^\leftarrow$.
We have
\[
\int_0^1 (f_1(x) - g(x))^2\,dx = \frac{W_2(X,Z)^2}{2} \quad\text{and}\quad \int_0^1 (f_2(x) - g(x))^2\,dx = \frac{W_2(Y,Z)^2}{2}.
\]
by Corollary~\ref{cor:W2_inv}.
Let $F,G \colon (0,1)^2 \to \R$ be given by
\[
F(x,y) = f_1(x) + f_2(y) \quad\text{and}\quad G(x,y) = g(x) + g(y).
\]
If we view $(0,1)^2$ as the sample space where area is interpreted as probability, then $F$ and $G$ are random variables having the same distributions as $(X+Y)/\sqrt{2}$ and $Z$.
Notice that
\begin{equation} \label{eq:W2_ineq}
\begin{split}
W_2\biggl(\frac{X+Y}{\sqrt{2}}, Z\biggr)^2 &\le
\int_0^1\int_0^1 (F(x,y) - G(x,y))^2 \,dx\,dy \\
&= \int_0^1 (f_1(x)-g(x))^2\,dx + \int_0^1 (f_2(y)-g(y))^2\,dy \\
&= \frac{W_2(X,Z)^2 + W_2(Y,Z)^2}{2}.
\end{split}
\end{equation}

Assume that the equality holds and let us show that $X$ and $Y$ are standard normal.
We claim that $G(x_1,y_1) = G(x_2,y_2)$ implies $F(x_1,y_1) = F(x_2,y_2)$.
To show this, suppose that $F(x_1,y_1) < F(x_2,y_2)$.
Since $f_1$ and $f_2$ are left-continuous, for some small $\eps > 0$ the squares $R_1 := (x_1-\eps,x_1] \times (y_1-\eps,y_1]$ and $R_2 := (x_2-\eps,x_2] \times (y_2-\eps,y_2]$ satisfy $\sup F(R_1) < \inf F(R_2)$.
Take small squares $S_1 = (a,a+\delta)\times(b,b+\delta) \subset R_1$ and $S_2 = (c,c+\delta)\times(d,d+\delta) \subset R_2$ such that $\inf G(S_1) > \sup G(S_2)$.
Let $\phi \colon S_1 \to S_2$ be given by $\phi(a+x,b+y) := (c+x,d+y)$ and define
\[
H(x,y) = \begin{cases}
F(\phi(x,y)) & \text{if $(x,y) \in S_1$,} \\
F(\phi^{-1}(x,y)) & \text{if $(x,y) \in S_2$, and} \\
F(x,y) & \text{otherwise.}
\end{cases}
\]

Since
\[
\begin{split}
\int_0^1 \int_0^1 &F(x,y)G(x,y)\,dx\,dy - \int_0^1 \int_0^1 H(x,y)G(x,y)\,dx\,dy \\
&= \int_0^\delta \int_0^\delta (F(a+x,b+y)-F(c+x,d+y))G(a+x,b+y) \,dx\,dy \\
&\quad + \int_0^{\delta} \int_0^{\delta} (F(c+x,d+y)-F(a+x,b+y))G(c+x,d+y) \,dx\,dy \\
&= \int_0^\delta \int_0^\delta (F(a+x,b+y)-F(c+x,d+y)) \\
&\quad\cdot(G(a+x,b+y)-G(c+x,d+y)) \,dx\,dy < 0,
\end{split}
\]
we have
\[
\int_0^1 \int_0^1 (H(x,y) - G(x,y))^2\,dx\,dy < \int_0^1 \int_0^1 (F(x,y) - G(x,y))^2\,dx\,dy.
\]
Since $H$ has the same distribution as $F$, the definition of $W_2$ tells us that the equality cannot hold in \eqref{eq:W2_ineq}, contrary to our assumption.
Thus, we have $F(x_1,y_1) = F(x_2,y_2)$ as claimed.

We now know that $F$ is constant on $G^{-1}(c)$ for all $c\in\R$.
Fix an $n \in \N$, and let $x_k \in \R$ be such that $g(x_k) = k/n$ for each $k \in \Z$.
Since $(x_k,x_{-k}) \in G^{-1}(0)$ and $(x_{k+1},x_{-k}) \in G^{-1}(1/n)$, the value of
\[ f_1(x_{k+1}) - f_1(x_k) = F(x_{k+1},x_{-k}) - F(x_k,x_{-k}) \]
is the same for all $k \in \Z$.
Notice that $g(x_{k+1}) - g(x_k)$ is the same for all $k\in\Z$, too.
Since this argument applies for all $n\in\N$, and $f_1$ is nondecreasing, we must have $f_1 = \alpha g + \beta$ for some $\alpha > 0$ and $\beta \in \R$.
As $\int f_1 = 0$ and $\int f_1^2 = 1$, the only possibility is $f_1 = g$. The same argument gives $f_2 = g$.
\end{proof}

\section{Proof of the Lindeberg--Feller CLT} \label{sec:lindeberg-feller}

\lindebergfeller*

\begin{proof}
Let $M := \lim_{\eps \to 0+} M_\eps$.
For each $n\in\N$, take independent mean-zero random variables $X_{n1},\dots,X_{nm_n}$ ($m_n \in \N$) with $|X_{nj}| \le 1/n$ for $j=1,\dots,m_n$ and $\sum_{j=1}^{m_n} \E [X_{nj}^2] = 1$ such that
\begin{equation} \label{eq:close_to_M}
W_2\Bigl(\sum_{j=1}^{m_n}X_{nj}, Z\Bigr) \ge M - 1/n.
\end{equation}

Since $\max_{j=1}^{m_n} \E [X_{nj}^2] \to 0$ as $n\to\infty$, we can choose $1 \le k_n \le m_n$ for each $n\in\N$ so that $v_n := \sum_{j=1}^{k_n} \E X_{nj}^2 \to 1/2$ as $n\to\infty$.
By passing to a subsequence if necessary, we may assume that
\[
\frac{1}{\sqrt{v_n}} \sum_{j=1}^{k_n} X_{nj} \tto L_1
\quad\text{and}\quad
\frac{1}{\sqrt{1-v_n}} \sum_{j=k_n+1}^{m_n} X_{nj} \tto L_2
\]
for some $L_1$ and $L_2$, by using Lemma~\ref{lem:conv_subseq}.

Since $|X_{nj}| \le 1$, we have $\E[X_{nj}^4] \le \E[X_{nj}^2]$, and thus
\[
\E\Bigl[\Bigl(\sum_{j=1}^{k_n} X_{nj}\Bigr)^4\Bigr]
\le \sum_{j=1}^{k_n} \E [X_{nj}^4] + 3\Bigl(\sum_{j=1}^{k_n}\E X_{nj}^2\Bigr)^2 \le v_n + 3v_n^2
\]
for all $n\in\N$.
By Lemma~\ref{lem:bdd_4_moment}, $L_1$ has mean $0$ and variance $1$.
The same holds for $L_2$.
By Proposition~\ref{prop:dist_implies_W2}, we have
\[
W_2(L_1,Z) = \lim_{n\to\infty} W_2\Bigl( \frac{1}{\sqrt{v_n}} \sum_{j=1}^{k_n} X_{nj}, Z\Bigr) \le M.
\]
Similarly we have $W_2(L_2,Z) \le M$.

Notice that we have
\[ \sum_{j=1}^{m_n} X_{nj} \tto \frac{L_1 + L_2}{\sqrt{2}}. \]
By Proposition~\ref{prop:dist_implies_W2} and \eqref{eq:close_to_M}, we have
\[
W_2\Bigl(\frac{L_1 + L_2}{\sqrt{2}}, Z\Bigr)
= \lim_{n\to\infty} W_2\Bigl( \sum_{j=1}^{m_n} X_{nj}, Z\Bigr) = M.
\]
Since
\[
M^2 = W_2\Bigl(\frac{L_1 + L_2}{\sqrt{2}}, Z\Bigr)^2 \le \frac{W_2(L_1,Z)^2 + W_2(L_2,Z)^2}{2} \le M^2,
\]
Theorem~\ref{thm:average_closer} implies that $L_1$ and $L_2$ are standard normal. Therefore, we have
\[ M = W_2(Z,Z) = 0. \qedhere\]
\end{proof}

\section{Measure theoretic probabilistic proofs} \label{sec:measure_proofs}

We can prove the following by directly using the definition of convergence in distribution.

\convsubseq*

\begin{proof}
Let $q \in \Q$.
By the Bolzano--Weierstrass theorem, some subsequence of $(F_{X_n}(q))_{n\in\N}$ converges to a finite number.
Using diagonaization, we can find $n_1 < n_2 < \dots$ such that
\[
\lim_{k\to\infty} F_{X_{n_k}}(q) = F_q \qquad \text{for some $F_q\in\R$}
\]
for all rational $q \in (0,1)$.

Now let $F \colon \R \to [0,1]$ be given by
\[ F(x) := \inf\{F_q :q \in (x,\infty)\cap\Q\}. \]
Then it can be shown that $F$ is a non-decreasing right continuous function such that $F_{n_k}(x) \to F(x)$ as $k\to\infty$ for all $x \in \R$ at which $F$ is continuous. For more details, see \cite[Lemma~9.6.2]{ResProb}.

By chebyshev's inequality, for each $M > 0$ we have
\[
F_{X_n}(-M) = \P(X_n \le -M) \le B/M^2
\quad\text{and}\quad 1 - F_{X_n}(M) = \P(X_n > M) \le B/M^2.
\]
This implies $F(-M) \le B/M^2$ and $F(M) \ge 1-B/M^2$, and thus we have $F(x) \to 1$ as $x \to \infty$ and $F(x) \to 0$ as $x \to -\infty$.
This shows that $F$ is a cumulative distribution function.
\end{proof}

An important property of inverses is that $X_n \tto X$ if and only if
\[ \inv{X_n}(t) \to \inv{X}(t) \qquad \text{as $n\to\infty$} \]
for all $t \in (0,1)$ where $\inv{X}$ is continuous. See \cite[Proposition~8.3.1 and Theorem~8.3.2]{ResProb} for a proof.
Replacing convergence in distribution with convergence with probability $1$ is useful when we prove facts like the following.

\momentconv*

\begin{proof}
Since we can replace $X_n$ and $X$ with $\inv{X_n}$ and $\inv{X}$, it is enough to assume $X_n \to X$ with probability $1$ instead of $X_n \tto X$.
Let $M > 0$ be such that $\P(X^2 = M) = 0$.
Since $X_n 1_{\{|X_n|\le M\}} \to X 1_{\{|X|\le M\}}$ on $\{|X| \ne M\}$, which has probability $1$, the bounded convergence theorem implies
\[
\E[X_n^2; X_n^2 \le M] \to \E[X^2; X^2 \le M].
\]
Here, $\E[X;\varphi(X)]$ denotes $\E[X1_{\{\varphi(X)\}}]$. Since
\[
\E[X_n^2; X_n^2 > M] \le \E[X_n^4/M^2] \le \frac{B}{M^2},
\]
letting $M \to \infty$ gives $\E[X_n^2] \to \E[X^2]$.
The proof for $\E[X_n] \to \E[X]$ is similar.
\end{proof}

Out interest in inverses in this note came from the following ``continuous version" of the rearrangement inequality.
The proof also resembles its discrete analogue.

\rearr*

%

\begin{proof}
First assume that $X$ and $Y$ are simple, i.e.\ there are $a_1,\dots,a_n,b_1,\dots,b_m \in \R$ such that $X \in \{a_1,\dots,a_n\}$ and $Y \in \{b_1,\dots,b_m\}$.
Let $X'$ and $Y'$ be random variables with the same distribution as $X$ and $Y$.
Since the set of possible $(\P(X' = a_i, Y' = b_j))_{i\le n, j\le m}$ is compact, the supremum of possible $\E[X'Y']$ is attained.

If $(X',Y')$ has a different joint distribution from $(\inv{X},\inv{Y})$, then there are $a,b,c,d \in \R$ with
\[ a < c,\quad b > d,\quad \P(X' = a, Y' = b) > 0,\quad \text{and}\quad \P(X' = c, Y' = d) > 0; \]
i.e.\ ``$X'$ doesn't increase as $Y'$ increases."
Now take events $E \subset \{X' = a, Y' = b\}$ and $F \subset \{X' = c, Y' = d\}$ with the same nonzero probability and swap the values of $Y'$ on $E$ and $F$ to form a new random variable $Y''$. (We can assume that the underlying sample space is good enough to do this.)
Since
\[
\E[XY'']-\E[XY'] = (ad + bc - ab - cd)\P(E) = (c-a)(b-d)\P(E) > 0,
\]
the supremum mentioned above is not attained by $(X',Y')$.
Since this is true whenever $(X',Y')$ has a different joint distribution from $(\inv{X},\inv{Y})$, the supremum must be attained by $(\inv{X},\inv{Y})$. This proves \eqref{eq:rearr}.

Now assume that $X$ and $Y$ are not necessarily simple.
We can take simple $X_1,X_2,\dots$ and $Y_1,Y_2,\dots$ such that
$|X_n| \le |X|$, $|Y_n| \le |Y|$, $X_n \to X$, and $Y_n \to Y$, $|\inv{X_n}| \le |\inv{X}|$, $|\inv{Y_n}| \le |\inv{Y}|$, $\inv{X_n} \to \inv{X}$, and $\inv{Y_n} \to \inv{Y}$.
Since $|X_nY_n| \le |XY|$, the dominated convergence theorem \cite[Theorem~5.3.3]{ResProb} gives $\E[X_nY_n] \to \E[XY]$.
Similarly we have $\int_0^1 \inv{X_n}\inv{Y_n} \to \int_0^1 \inv{X}\inv{Y}$.
As $ \E[X_nY_n] \le \int_0^1 \inv{X_n}\inv{Y_n}$, letting $n \to \infty$ gives \eqref{eq:rearr}.
\end{proof}

Recall the following corollary.

\wasschar*

Using this characterization of $2$-Wasserstein metric, we can prove the relation between $W_2$ and the convergence in distribution.

\wimpliesdist*
\begin{proof}
Assume that $X_n \not\tto X$. Then there is a $t \in (0,1)$ where $\inv{X}$ is continuous such that $\inv{X_n}(t) \not\to \inv{X}(t)$.
Let $\eps > 0$ and $n_1 < n_2 < \dots$ be such that
\[
\bigl|\inv{X_{n_k}}(t) - \inv{X}(t)\bigr| \ge 2\eps \qquad \text{for all $k\in\N$.}
\]
Since $\inv{X}$ is continuous at $t$, there is a $\delta > 0$ such that $\bigl|\inv{X}(s) - \inv{X}(t)\bigr| \le \eps$ for all $s \in (t-\delta,t+\delta)$.

Let $k\in\N$. If $\inv{X_{n_k}}(t) \le \inv{X}(t) - 2\eps$, then on $(t-\delta,t]$ we have $\inv{X} - \inv{X_{n_k}} \ge \eps$.
If $\inv{X_{n_k}}(t) \ge \inv{X}(t) + 2\eps$, then on $[t,t+\delta)$ we have $\inf{X_{n_k}} - \inv{X} \ge \eps$.
In either case, we have
\[
\int_0^1 \bigl( \inv{X_{n_k}} - \inv{X} \bigr)^2 \ge \eps^2\delta.
\]
Since this is true for all $k \in \N$, we have $W_2(X_n,X) \not\to 0$.
\end{proof}

\distimpliesw*
\begin{proof}
Since we can replace $X_n$, $X$, and $Y$ with $\inv{X_n}$, $\inv{X}$, and $\inv{Y}$, it is enough to show that if $X_n \to X$ with probability $1$, then $\E[(X_n-Y)^2] \to \E[(X-Y)^2].$
As $\E X_n^2 = \E X^2 = \E Y^2 = 1$, it suffices to show $\E[(X_n-X)Y] \to 0$.
Since
\[
\bigl|\E[(X_n-X)Y]\bigr|^2 \le \E\bigl[(X_n-X)^2\bigr]\E[Y^2]
\]
by the Cauchy-Schwarz inequality, showing $\E\bigl[(X_n-X)^2\bigr] \to 0$ is enough.

Let $\eps > 0$ be given.
Choose $M > 1$ with $\E\bigl[X^2; |X|>M-1\bigr] \le \eps$.
Let $\phi \colon \R \to \R$ be the continuous function that satisfies $\phi(x) = x^2$ if $|x| \le M-1$, $\phi(x) = 0$ if $|x| \ge M$, and is linear on the remaining intervals.
By the bounded convergence theorem, we have
\[ \lim_{n\to\infty} \E[\phi(X_n)] = \E[\phi(X)] \ge 1-\eps, \]
and thus $\limsup_{n\to\infty} \E\bigl[X_n^2; |X_n|>M\bigr] \le \eps$.

Since $(X_n-X)^2 \le 2(X_n^2 + X^2)$, we have
\[
\begin{split}
\E\bigl[(X_n-X)^2; |X_n - X| > 2M\bigr] 
&\le 2\E[X_n^2 + X^2; |X_n| \ge |X|, |X_n| > M] \\
&\quad+ 2\E[X_n^2 + X^2; |X_n| < |X|, |X| > M] \\
&\le 4\E[X_n^2; |X_n|>M] + 4\E[X^2; |X|>M].
\end{split}
\]
This implies
\[
\limsup_{n\to\infty} \E\bigl[(X_n-X)^2;|X_n-X|>2M\bigr] \le 8\eps.
\]
On the other hand, the bounded convergence theorem implies
\[\E\bigl[(X_n-X)^2; |X_n-X| \le 2M\bigr] \to 0.\]
Combining the two pieces together and letting $\eps \to 0$ finishes the proof.
\end{proof}

\bibliographystyle{alpha}
\bibliography{references}
\end{document}